\documentclass{amsart}
\usepackage{amssymb,amsmath, amsfonts, amsrefs, tikz, epsfig, float}
\usepackage{amsthm}
\usepackage{mathrsfs}
\usepackage{eucal}
\usepackage{enumerate, indentfirst}
\usepackage[fleqn,tbtags]{mathtools}
\usepackage{color}

\usepackage{multicol}
 \newtheorem{theorem}{Theorem}[section]%
 \newtheorem{corollary}[theorem]{Corollary}
 
 \newtheorem{proposition}[theorem]{Proposition}

 \theoremstyle{definition}
 \newtheorem{example}{Example}[section]
 
 \theoremstyle{remark}
  \numberwithin{equation}{section}

\renewcommand{\phi}{\varphi}
\renewcommand{\theta}{\vartheta}

\DeclareMathOperator{\tform}{\mathfrak{t}}

\DeclareMathOperator{\wform}{\mathfrak{w}}

\DeclarePairedDelimiterX\sipt[2]{(}{)_{\tform}}{#1\,\delimsize\vert\,#2}
\DeclarePairedDelimiterX\sipv[2]{(}{)_{v}}{#1\,\delimsize\vert\,#2}
\DeclarePairedDelimiterX\sipw[2]{(}{)_{w}}{#1\,\delimsize\vert\,#2}

\renewcommand{\le}{\mathscr L(E;E)}

\newcommand{\pair}[2]{\left(\begin{array}{c}\!\!#1\!\!\\ \!\!#2\!\!\end{array}\right)}

\newcommand{\alg}{\mathscr{A}}

\newcommand{\ideal}{\mathscr{I}}

\newcommand{\dupR}{\mathbb{R}}
\newcommand{\dupC}{\mathbb{C}}

\newcommand{\dom}{\operatorname{dom}}

\newcommand{\ran}{\operatorname{ran}}

\newcommand{\lef}{\mathscr{L}(E;F)}

\newcommand{\bha}{\mathscr{B}(\hila)}

\newcommand{\lee}{\mathscr{L}(E;\anti{E})}

\newcommand{\M}{\mathscr{I}}
\newcommand{\sigef}{\sigma(E,F)}
\newcommand{\sigfe}{\sigma(F,E)}
\newcommand{\AmaxB}{A_{max}^B}
\newcommand{\AmaxM}{A_{max}^M}
\newcommand{\hil}{H}

\newcommand{\hila}{H_A}

\newcommand{\kil}{K}
\newcommand{\bh}{\mathscr{B}(\hil)}

\DeclarePairedDelimiterX\abs[1]{\lvert}{\rvert}{#1}
\DeclarePairedDelimiterX\sip[2]{(}{)}{#1\,\delimsize\vert\,#2}
\DeclarePairedDelimiterX\sipta[2]{(}{)_{\!_{\widetilde{A}}}}{#1\,\delimsize\vert\,#2}
\DeclarePairedDelimiterX\sipf[2]{(}{)_{f}}{#1\,\delimsize\vert\,#2}
\DeclarePairedDelimiterX\sipg[2]{(}{)_{g}}{#1\,\delimsize\vert\,#2}
\DeclarePairedDelimiterX\siptw[2]{(}{)_{\tform+\wform}}{#1\,\delimsize\vert\,#2}
\DeclarePairedDelimiterX\set[2]{\{}{\}}{#1\,\delimsize\vert\,#2}
\DeclarePairedDelimiterX\dual[2]{\langle}{\rangle}{#1,#2}
\DeclarePairedDelimiterX\duale[2]{\langle}{\rangle_1}{#1,#2}
\DeclarePairedDelimiterX\dualk[2]{\langle}{\rangle_2}{#1,#2}
\DeclarePairedDelimiterX\sdual[2]{[}{]}{#1,#2}
\DeclarePairedDelimiterX\sipa[2]{(}{)_{\!_A}}{#1\,\delimsize\vert\,#2}
\DeclarePairedDelimiterX\sipc[2]{(}{)_{\!_C}}{#1\,\delimsize\vert\,#2}
\DeclarePairedDelimiterX\sipab[2]{(}{)_{\!_{A+B}}}{#1\,\delimsize\vert\,#2}
\DeclarePairedDelimiterX\sipb[2]{(}{)_{\!_B}}{#1\,\delimsize\vert\,#2}
\newcommand{\anti}[1]{\bar{#1}'}

\newcommand{\opmatrix}[4]{\left(\begin{array}{cc}\!\! #1 &  #2\!\!\\ \!\! #3& #4\!\!\end{array}\right)}

\newcommand{\ta}{\widetilde{A}} %jel
\allowdisplaybreaks
\begin{document}
\sloppy 
\title[Generalized Krein--von Neumann extension]{Operators on anti-dual pairs:\\ Generalized Krein--von Neumann extension}

\author[Zs. Tarcsay]{Zsigmond Tarcsay 
}
\thanks{The corresponding author Zs. Tarcsay (tarcsay@cs.elte.hu) was supported by the ``For the Young Talents of the Nation'' scholarship program (NTP-NFT\"O-17) of the Hungarian Ministry of Human
Capacities.}

\address{%
Zs. Tarcsay \\ Department of Applied Analysis  and Computational Mathematics\\ E\"otv\"os Lor\'and University\\ P\'azm\'any P\'eter s\'et\'any 1/c.\\ Budapest H-1117\\ Hungary}

\email{tarcsay@cs.elte.hu}

\author[T. Titkos]{Tam\'as Titkos}
\thanks{T. Titkos was supported by the Hungarian National Research, Development and Innovation Office -- NKFIH (PD128374), and by the
J\'anos Bolyai Research Scholarship of the Hungarian Academy of Sciences.}

\address{T. Titkos \\ Alfr\'ed R\'enyi Institute of Mathematics, Hungarian Academy of Sciences\\ Re\'altanoda utca 13-15.\\ Budapest H-1053\\ Hungary\\ and BBS University of Applied Sciences\\ Alkotm\'any u. 9.\\ Budapest H-1054\\ Hungary }
\email{titkos.tamas@renyi.mta.hu}

%----------classification, keywords, date


\begin{abstract}  The main aim of this paper is to generalize the classical concept of positive operator, and to develop a general extension theory, which overcomes not only the lack of a  Hilbert space structure, but also the lack of a normable topology. The concept of anti-duality carries an adequate structure to define positivity in a natural way, and is still general enough to cover numerous important areas where the Hilbert space theory cannot be applied. Our running example -- illustrating the applicability of the general setting to spaces bearing poor geometrical features -- comes from noncommutative integration theory. Namely, representable extension of linear functionals of involutive algebras will be governed by their induced operators. The main theorem, to which the vast majority of the results is built, gives a complete and constructive characterization of those operators that admit a continuous positive extension to the whole space. Various properties such as commutation, or minimality and maximality of special extensions will be studied in detail.

\end{abstract}

\maketitle

\section{Introduction}

One of the most natural questions that arises when someone dealing with partially defined objects in mathematics is whether there exists an extension sharing some prescribed properties. A great many authors have studied abstract extension problems for operators on Hilbert spaces, that go at least back to M. G. Krein \cite{Krein} and J. von Neumann \cite{JvN}.  (For various different developments of their groundbreaking work see e.g. \cites{AG,AGLMS,AGMST,HMdS,HSdSW,Skau,SSZT,BN}, and the references therein.) The following extension problem was posed by P. R. Halmos \cite{Halmos}: Assume that a positive operator $A:D\to H$ is given, where $D$ is a linear subspace of the complex Hilbert space $\hil$. Positivity here means that $\sip{Ax}{x}\geq0$ for all $x\in D$. The question is: under what conditions can we guarantee the existence of an everywhere defined bounded positive extension $\widetilde{A}$ of $A$? Of course, if there is any then $A$ itself must be  bounded. Hence, extending it to the closure by continuity, we may suppose that $D$ is closed. Consider the matrix representation of $A$ with respect to the orthogonal decomposition $\hil=D\oplus D^\perp$
\begin{align*}
[A]=\opmatrix{A_{11}}{*}{A_{21}}{*},
\end{align*}
where $A_{11}:D\to D$ and $A_{21}:D\to D^{\perp}$ arise in the usual way,  whereas the second column (of symbols $*$) waits to be filled to obtain a positive operator. It is easy to see that every positive extension of $A$ has representation of the form 
\begin{align}\label{E:opmatrix}
[\widetilde{A}]=\opmatrix{A_{11}}{A_{21}^*}{A_{21}}{X},
\end{align}
where $A_{11}:D\to D$ and $X:D^{\perp}\to D^{\perp}$ are positive. So, extending $A$ to a positive operator $\ta$ is equivalent to find $X\geq0$ such that the operator matrix \eqref{E:opmatrix} is positive. (For a more general completion problem for block operators see \cite{BH}.) This form helps us to demonstrate also that such an extension need not exist even in the most simple case. Indeed, assume that $\hil$ is of the form $\hil=\kil\oplus\kil$ with a complex Hilbert space $K$ and assume that $A_{11}$ is the zero operator, while $A_{21}=A_{21}^*$ is any positive but nonzero bounded operator on $K$. Then an elementary calculation shows that there is no positive completion at all.

The main aim of this paper is to develop a general extension theory  which overcomes the problem of not having orthogonal decomposition when we drop the Hilbert space structure. This level of generality is indeed necessary in our considerations, because we intend to investigate extendibility of ``positive'' mappings acting on spaces without inner product. In order to introduce the appropriate analogues of standard operator classes, that cover the original Hilbert space setting, and is general enough to be applicable for objects like operator kernels and representable functionals, we are going to consider anti-dual pairs. That is, two appropriately chosen complex linear spaces intertwined by a sesquilinear and separating map, called anti-duality. This is a slight modification of the well known dual pair setting (see \cite{Schaefer}*{Chapter IV}), the only difference we make is the conjugate linearity in the second variable. Having an anti-duality at hand also allows us to introduce various topologies, and hence continuity and boundedness of maps acting between the spaces in question.

The organization of the paper is as follows. Section \ref{S: Preliminaries} contains a short overview of concepts that are needed in later sections of the paper. In particular, we are going to introduce the notions of positivity and symmetry of operators in context of anti-dual pairs. It will turn out that these operators are weakly continuous, which suggests that the most adequate topologies for our investigations are the weak topologies.

In Section \ref{S: MT} we present a quite general extension theorem, which can be considered as the main result of the paper. In fact, this result will serve as the base of our further investigations throughout. Roughly speaking, we are going to characterize (in both topological and algebraic ways) operators possessing positive extension to the whole space. In addition, the construction which is based on the ingenious paper of  Z. Sebesty\'en \cite{Sebestyen93} has some useful theoretical consequences, including an explicit formula for the obtained positive extension, as well as for its ``quadratic form''. Also, it will turn out that this extension is minimal in some sense, thus we shall call it Krein-von Neumann extension, in accordance with the classical literature. We emphasize that in this paper we restrict ourselves to continuous extensions, and thus the operators we are going to deal with are typically not densely defined. Hence, the maximal (Friedrichs) extension, apart from trivial cases, does not exist. Nevertheless, we will show that the set of positive extensions bounded by a fixed positive operator always possesses a maximal element. Closing that chapter we will also show how the abstract result can be used for positive definite operator functions. The special case when the anti-duality is the evaluation on the pair of a fixed Banach space and its conjugate topological dual will be considered in Section \ref{S: Banach}. On the one hand, we will obtain a strengthening of the main result of \cite{SSZT}, on the other hand, we will conclude that Halmos' original result follows indeed as an immediate consequence of our main theorem.
Finally, we apply our results to obtain representable positive extensions of linear functionals given on a left ideal of an involutive algebra.

\section{Preliminaries}\label{S: Preliminaries} 
In this section we collect all the necessary ingredients from the theory of topological vector spaces, anti-dual pairs, and their special linear operators which we are going to use throughout.
\subsection{Anti-dual pairs and related topologies} We start by recalling the notion of anti-duality, which is just a slight (and technical) modification of dual pairing. Although there is no crucial difference between these two notions, we choose anti-duality in order to stay formally as close as possible to the Hilbert space case. Let $E, F$ be complex vector spaces. A function $\dual{\cdot}{\cdot}:F\times E\to \dupC$ is called anti-duality if it is sesquilinear (that is to say, linear in the first argument and conjugate linear in the second one) and $\dual{\cdot}{\cdot}$ separates the points of $F$ and $E$ (that is to say, if $\dual{f}{x}=0$ for all $x\in E$ then $f=0_F$ and   if $\dual{f}{x}=0$ for all $f\in F$ then $x=0_E$). The triple $((E,F),\dual{\cdot}{\cdot})$ is called an anti-dual pair, and it is denoted shortly by $\dual{F}{E}$. Observe that if $\dual{F}{E}$ is an anti-dual pair then $\dual{E}{F}'$ is also an anti-dual pair where $\dual{\cdot}{\cdot}':E\times F\to\dupC$ is given by
\begin{equation}\label{E:dualvesszo}
    \dual{x}{f}':=\overline{\dual{f}{x}},\qquad x\in E,~ f\in F.
\end{equation}

The most natural anti dual pair is a linear space and a linear subspace of its conjugate algebraic dual, intertwined by the evaluation as anti-duality. In fact, every anti-dual pair can be written in the above form. Indeed, if $\dual{F}{E}$ is an anti-dual pair, then due to the identification \begin{equation*}
x\mapsto\varphi_x;\qquad\varphi_x(f):= \dual{f}{x}\qquad\mbox{for all}\quad f\in F,
\end{equation*}
$E$ may be regarded as a linear subspace of $F^*$, the algebraic dual of $F$. Similarly,  due to the mapping  
\begin{equation*}
    f\mapsto\psi_f;\qquad\psi_f(x):=\dual{f}{x}\qquad\mbox{for all}\quad x\in E,
\end{equation*}
$F$ can be identified as a linear subspace of $\bar{E}^*$, the algebraic anti-dual space of $E$. Our prototype of anti-dual pairs is the system $((H,H),\sip{\cdot}{\cdot})$ where $H$ is a Hilbert space with inner product $\sip{\cdot}{\cdot}$. This particular anti-dual pair has the useful feature that $H$ can be identified with its topological dual along the maps $x\mapsto\varphi_x$ and $f\mapsto\psi_f$, according to the Riesz representation theorem.  A similar feature is obtained in the general setting if we endow $E$ and $F$ with appropriate topologies.  For this purpose the most natural at hand are the \emph{weak topologies} $\sigef$ and $\sigfe$ on $E$ and $F$, respectively: $\sigef$ is the smallest topology making $\varphi_x$ continuous for all $x\in E$, and similarly, $\sigfe$ is the smallest topology on $F$ such that all the functionals of the form $\psi_f$ $(f\in F)$ are continuous. Both $(E,\sigef)$ and $(F,\sigfe)$ are locally convex Hausdorff spaces such that $\bar{E}'=F$ and $F'=E$, where $\bar{E}'$ and $F'$ refer to the topological anti-dual and dual spaces of $E$ and $F$, respectively. We call an anti dual pair $\dual FE$ $w^*$-sequentially complete, if $F$ is $\sigfe$ sequentially complete. In order to demonstrate that being $w^*$-sequentially complete does not mainly depend on the spaces but on the way of how to intertwine them, we cite an example of Wilansky. 

\begin{example}
Let us denote by $c$ and $\ell^1$ the spaces of convergent and absolute summable complex sequences, respectively.  Define an anti-duality by
\begin{equation*}
    \duale xy:=\sum\limits_{n=1}^{\infty}{x}(n)\overline{{y}(n)},\qquad x\in \ell^1,~y\in c.
\end{equation*}
Clearly, $\duale{\ell^1}{c}$ is an anti-dual pair, which is not $w^*$-sequentially complete, as the sequence of Kronecker delta's ($\delta^n(n)=1$ and $\delta^n(k)=0$ if $k\neq n$) is a $w^*$-Cauchy but not $w^*$-convergent sequence in $\ell^1$. On the other hand, the anti-dual pair $\dualk{\ell^1}{c}$ with
\begin{equation*}
    \dualk xy:=x(1)\cdot\overline{\lim_{n\to\infty}y(n)}+\sum\limits_{n=1}^{\infty}{x}(n+1)\overline{{y}(n)}
\end{equation*}
is $w^*$-sequentially complete according to the Banach-Steinhaus theorem  (for the details see Example 7 of 8-2 and Problem 108-109 of 8-2 in \cite{Wilansky}).
\end{example}
We remark also here that a linear space and its algebraic conjugate dual intertwined by evaluation is always $w^*$-(sequentially) complete. A Banach space and its topological conjugate dual form also a $w^*$-sequentially complete anti-dual pair, according to the Banach-Steinhaus theorem. More generally,  a barrelled space and its topological conjugate dual form a quasi-complete, and hence a $w^*$-sequentially complete anti-dual pair. Nevertheless we mention that the   a Banach space is $w^*$-complete if and only if the space is finite dimensional.

Besides the weak topology, the strong topology will be particularly important for our investigations. Recall that if $\dual{F}{E}$ is an anti-dual pair, then the polar of $A\subseteq F$ is the absolutely convex set
\begin{equation*}
    A^{\circ}=\set[\big]{x\in E}{\forall a\in A:~|\dual{a}{x}|\leq 1}.
\end{equation*}
The strong topology $\beta(E,F)$ on $E$ is obtained as the linear topology induced by polars of $\sigfe$-bounded subsets of $F$. Similarily, the strong topology $\beta(F,E)$ on $F$ is induced by polars of $\sigef$-bounded subsets of $E$. Now we turn to investigate special linear operators acting between two sides of anti-dual pairs.

\subsection{Positive- and symmetric operators on anti-dual pairs}
 If an anti-dual pair $\dual{F}{E}$ is given, we will use the short notation $A:E\supseteq \dom A\to F$ for \emph{linear} operators acting on a subspace $\dom A$ of $E$ with values in $F$. We prefer this setting instead of considering duality with conjugate linear operators as in \cite{Alpay}. An operator $A:E\supseteq\dom A\to F$ is called \emph{positive} if it satisfies
\begin{equation*}
    \dual{Ax}{x}\geq0\qquad \textrm{for all $x\in\dom A$},
\end{equation*}
and \emph{symmetric} if 
\begin{equation*}
    \dual{Ax}{y}=\overline{\dual{Ay}{x}}\qquad \textrm{for all $x,y\in\dom A$}.
\end{equation*}
It is obvious that these are direct generalizations of the well-known notions of Hilbert space theory. The main advantage is that this setting allows us to handle structures without Hilbert space structure analogously. As the next example will illustrate, operators on anti-dual pairs appear very naturally for example in noncommutative integration theory.

\begin{example}
Let $\alg$ be a  ${}^*$-algebra with algebraic conjugate-dual $\bar{\alg}^*$, and let $\M\subseteq\alg$ be a left-ideal. Then $\dual{\bar{\alg}^*}{\alg}$ is an anti-dual pair with $\dual{f}{a}:=f(a)$. If a positive linear functional $f:\M\to\mathbb{C}$ is given, we can associate a positive operator $A_f:\M\to\bar{\alg}^*$ as $\dual{A_fa}{b}:=f(b^*a)$. It will turn out later that positive extendibility of $f$ to the whole algebra can be characterized by means of $A_f$. Furthermore, the canonical extension of $f$ itself will be gained from the canonical extension of $A_f$. 
\end{example}
Recall one of the main advantages of weak topology, namely that a linear operator $T$ acting on a topological vector space $(V,\mathscr{T}_V)$ with values in $(F,\sigfe)$ is continuous if and only if the linear functionals
\begin{equation}\label{F: adj}
    \vartheta_x:V\to\dupC;\qquad\vartheta_x(h):=\dual{Th}{x}
\end{equation}
are continuous for each $x\in E$. For the sake of simplicity we introduce the following terminology: if two anti-dual pairs $\duale{F_1}{E_1}$ and $\dualk{F_2}{E_2}$ are given, we will call a map $T : E_1 \to E_2$ \emph{weakly continuous} if $T$ is $\sigma(E_1,F_1)$-$\sigma(E_2,F_2)$ continuous. The set of weakly continuous linear operators $T:E_1\to E_2$ is denoted by $\mathscr{L}(E_1;E_2)$. By replacing $\dualk{F_2}{E_2}$ with $\dualk{E_2}{F_2}'$ (see \eqref{E:dualvesszo}) one can characterize weak continuity of an operator $T:E_1\to F_2$. Indeed, according to \eqref{F: adj}, $T$ is weakly continuous if and only if for all $x_2\in E_2$ there exists  $f_1\in F_1$ such that
\begin{equation}\label{F: adj identity}
    \dualk{Tx_1}{x_2}=\overline{\duale{f_1}{x_1}}~(=\duale{x_1}{f_1}')\qquad\mbox{for all}\quad x_1\in E_1.
\end{equation}
The (necessarily weakly continuous) operator $T^*:E_2\to F_1$ satisfying \begin{equation*}%\label{E:w-adjoint}
    \dualk{Tx_1}{x_2}=\overline{\duale{T^*x_2}{x_1}},\qquad  x_1\in E_1,x_2\in E_2.
\end{equation*} is called the adjoint of $T$. In particular, the adjoint $A^*$ of an operator $A\in\lef$ belongs again to $\lef$ and satisfies $\dual{Ax}{y}=\overline{\dual{A^*y}{x}}$ for all $x,y\in E$.
Hence it makes sense to speak about self-adjointness $A=A^*$ of an operator $A\in\lef$. An everywhere defined symmetric operator (that is, an operator $S:E\to F$ such that $\dual{Sx}{y}=\overline{\dual{Sy}{x}}, x,y\in E)$ is automatically weakly continuous, and hence self-adjoint. If $A:E\supseteq \dom A\to F$ is an  operator such that $\dual{Ax}{x}$ is real for all $x\in \dom A$ then the sesquilinear form $\tform_A(x,y):=\dual{Ax}{y}$ ($x,y\in \dom A$) is hermitian, thus  $A$ is symmetric. Indeed, $\dual{Ax}{y}=\tform_A(x,y)=\overline{\tform_A(y,x)}=\overline{\dual{Ay}{x}}$ holds for all $x,y\in \dom A$. Let us summarize these observations in the following proposition.
\begin{proposition}\label{P:Hellinger}
    Let $\dual FE$ be an anti-dual pair.
    \begin{enumerate}[\upshape (a)]
      \item Every symmetric operator $A:E\to F$ is weakly continuous  and self-adjoint.
      \item  An operator $A:E\to F$ is symmetric if and only if $\dual{Ax}{x}$ is real for each $x\in E$.
      \item Every positive operator $A:E\to F$ is weakly continuous and self-adjoint.
    \end{enumerate}
\end{proposition}
As our main interest in this paper lies in positive extensions, we present a prototype of positive operators.
In rest of the paper we are mainly interested in positive operators on an anti-dual pair. In the next example we present the prototype of such positive operators.
\begin{example}\label{Ex:prototype}
    Let $\dual FE$ be an anti-dual pair and $\hil$ a Hilbert space. If $T:E\to\hil$ is a $\sigma(E,F)$-$\sigma(H,H)$ continuous linear operator then its adjoint $T^*:\hil\to F$ is $\sigma(H,H)$-$\sigma(F,E)$ continuous so that $T^*T\in\lef$ is positive:
    \begin{equation*}
        \dual{T^*Tx}{x}=\sip{Tx}{Tx}\geq0,\qquad x\in E.
    \end{equation*}
\end{example}
We will see later that, under some natural conditions on $F$, each positive operator $A\in\lef$ possesses a factorization of the form $A=T^*T$ with a suitable $T$ and $H$ of Example \ref{Ex:prototype}.

\section{Main theorem -- Extensions of positive operators} \label{S: MT}

The central problem of this section is to provide necessary and sufficient conditions under which a linear operator $A:E\supseteq \dom A\to F$ possesses a positive extension to the whole $E$. The following set associated to $A$ will play a key role in our treatment:
\begin{equation}
W(A):=\set{Ax}{x\in \dom A, \dual{Ax}{x}\leq 1}\subseteq F.
\end{equation}
The construction below is motivated by the work of Sebesty\'en \cite{Sebestyen93}. Here we emphasize that no topological conditions on $\dom A$ are assumed.
\begin{theorem}\label{T:Krein-von Neumann}
Let $\dual{F}{E}$ be a $w^*$-sequentially complete anti-dual pair and let $A:E\supseteq\dom A\to F$ be a linear operator with domain $\dom A$, which is assumed to be only a linear subspace. Then the following statements are equivalent.

%For a linear operator $A: \dom A\to F$ the %following assertions are equivalent. 
\begin{enumerate}[\upshape (i)]
\item There is a positive operator $\ta\in\lef$ extending $A$,
\item $W(A)$ is $\beta(F,E)$-bounded in $F$,
\item $W(A)$ is $\sigma(F,E)$-bounded in $F$,
\item To any $y$ in $E$ there is $M_y\geq0$ such that 
\begin{equation}\label{E:M_y}
 \abs{\dual{Ax}{y}}^2\leq M_y\dual{Ax}{x}\qquad \textrm{for all $x\in\dom A$.}
\end{equation}
\end{enumerate}
If one (and hence all) of the above conditions is satisfied then there exists a distinguished extension called the Krein-von Neumann extension
which is minimal in the following sense: $A_N\leq \widetilde{A}$ holds for any positive extension $\widetilde{A}\in\lef$ of $A$.
\end{theorem}
\begin{proof} We start by proving that (iv) implies (i). Let us consider the following inner product on the range space $\ran {A}$ of ${A}$: 
\begin{equation}
\sipa{{A}x}{{A}x'}:=\dual{{A}x}{x'},\qquad x,x'\in\dom {A}.
\end{equation}
First of all note that $\sipa{\cdot}{\cdot}$ is well defined: if $Ax=Ax'$ and $Ay=Ay'$, then
\begin{align*}
\dual{Ax}{y}=\dual{Ax'}{y}=\overline{\dual{Ay}{x'}}=\overline{\dual{Ay'}{x'}}=\dual{Ax'}{y'}.
\end{align*}
Furthermore, if $\dual{{A}x}{x}=0$ for some $x\in\dom {A}$ then (iv) implies $\dual{{A}x}{y}=0$ for all $y\in E$, hence ${A}x=0$, which means that $(\ran{A},\sipa{\cdot}{\cdot})$ is a pre-Hilbert space. Consider now the completion $\hil_{A}$ of $\ran {A}$ equipped with that inner product. Let us consider the densely defined canonical embedding operator
\begin{equation}\label{E:J}
J:H_A\supseteq\ran A\to F;\quad J({A}x):={A}x,\qquad x\in\dom {A}.
\end{equation}
For any $y$ in $E$, we see by (iv) that 
\begin{equation*}
\abs{\dual{{A}x}{y}}^2\leq M_y\sipa{{A}x}{{A}x},\qquad x\in\dom {A},
\end{equation*}
hence the linear functional
\begin{align*}
\hil_{A}\supseteq \ran {A}\ni {A}x\mapsto \dual{{A}x}{y}, 
\end{align*}
is continuous with respect to the norm induced by $\sipa{\cdot}{\cdot}$. That means that $J$ is $\sigma(\ran {A},\hil_{A})$-$\sigfe$ continuous and thus it admits a unique continuous extension to $\hila$ by $w^*$-sequentially completeness of $F$. As it will not cause ambiguity, this extension will be denoted by the same symbol $J\in\mathscr{L}(\hil_{A} ;F)$. The adjoint operator $J^*\in\mathscr{L}(E ;\hil_{A})$ fulfills 
\begin{align*}
\sipa{{A}x'}{J^*x}=\dual{J({A}x')}{x}=\dual{{A}x'}{x}=\sipa{{A}x'}{{A}x}
\end{align*}
for every $x,x'\in\dom {A}$, whence we gain the useful identity 
\begin{equation}\label{E:J*}
J^*x={A}x,\qquad x\in\dom {A}.
\end{equation}
And here we have arrived: for any $x$ in $\dom {A}$ we find that $JJ^*x=J({A}x)={A}x$, hence  $JJ^*\in\lef$ is a positive extension of $A$.
The extension obtained by this construction will be denoted by $A_N$.
 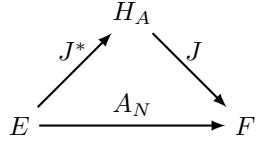
\begin{figure}[H]
\begin{tikzpicture}
\draw (10,0) node (A) {$E$};
\draw (13,0) node (B) {$F$};
\draw[thick,-latex] (A) -- (B)
 node[pos=0.5,above,sloped] {$A_N$};
\draw (11.5,1.5) node (C) {$\hila$};
\draw[thick,-latex] (A) -- (C)
 node[pos=0.5,above] {$J^*$\,};
\draw[thick,-latex] (C) -- (B)
 node[pos=0.5,above] {~$J$};
\end{tikzpicture}
\caption{Factorization
of the Krein-von Neumann extension.}
\end{figure}
If $A\in\lef$ is any positive extension of $A$ then clearly $  W(A) \subseteq  W(\ta)$. Hence, to see that (i) implies (ii) we may assume that $A$ is everywhere defined, i.e., $A\in\lef$. We have to show that  $ W(A)$ is bounded in the $\beta(F,E)$ topology, that is, $B^\circ$ absorbs $ W(A)$  for each $\sigma(E,F)$-bounded subset $B$ of $E$.  Note that we have $A=JJ^*$ and that $W(A) \subseteq J\langle B_{A}\rangle$, where $B_{A}$ refers to the closed unit ball in $\hil_{A}$. Since $B_{A}$ is weakly compact, $J\langle B_{A}\rangle$  is absorbed by the barrel $B^{\circ}$,  according to the absorbing lemma.  Hence $B^{\circ}$ absorbs $ W(A)$ as well, proving that (i)$\Rightarrow$ (ii). 

As the topology $\beta(F,E)$ is stronger then $\sigfe$, (ii) implies (iii) obviously.

Assume now that $W(A)$ is weak-* bounded in $F$ and take $y$ from $E$. By continuity we see that $\dual{\cdot}{y}$ is bounded on $ W(A)$ whence we conclude  that
\begin{align*}
M_y:=\sup\set[\big]{\abs{\dual{Ax}{y}}^2}{x\in \dom A, \dual{Ax}{x}\leq 1}<+\infty,
\end{align*}
which clearly implies (iv). The equivalence of (i)-(iv) has been proved.

In order to prove minimality of $A_N$, we are going to show first that
\begin{align}\label{E:ANyy-2}
\dual{A_Ny}{y}=\sup\set[\big]{\abs{\dual{Ax}{y}}^2}{x\in\dom A, \dual{Ax}{x}\leq 1},\qquad y\in E.
\end{align}
Indeed, using density of $\ran A$ in $\hila$, we conclude for all $y\in E$ that 
\begin{align*}
\dual{JJ^*y}{y}=\sipa{J^*y}{J^*y} 
&=\set[\big]{\abs{\sipa{Ax}{J^*y}}^2}{x\in\dom A, \sipa{Ax}{Ax}\leq 1}\\
&=\sup\set[\big]{\abs{\dual{Ax}{y}}^2}{x\in\dom A, \dual{Ax}{x}\leq 1}.
\end{align*}
 Consider now a positive extension $\widetilde{A}\in\lef$ of $A$. Then
\begin{align*}
\dual{JJ^*y}{y}&=\sup\set[\big]{\abs{\dual{Ax}{y}}^2}{x\in\dom A, \dual{Ax}{x}\leq 1}\\
               &\leq \sup\set[\big]{\abs{\dual{\widetilde{A}x}{y}}^2}{x\in E, \dual{\widetilde{A}x}{x}\leq 1}\\
               &\leq \sup\set[\big]{\dual{\widetilde{A}x}{x}\dual{\widetilde{A}y}{y}}{x\in E, \dual{\widetilde{A}x}{x}\leq 1}\\
               &=\dual{\widetilde{A}y}{y},
\end{align*}
holds for any $y\in E$. Hence $JJ^*\leq  \widetilde{A}$, indeed.  
\end{proof}
For the sake of brevity and clarity we introduce now a short term to express that  $\dual{F}{E}$ is a given anti-dual pair where the topological space $(F,\sigfe)$ is sequentially complete, $\dom A$ is a linear subspace of $E$, and $A:\dom A\to F$ is a linear operator that satisfies one (and hence all) of the properties (i)-(iv) in Theorem \ref{T:Krein-von Neumann}. Following Halmos' terminology we will say shortly that \emph{$A$ is a subpositive operator on the $w^*$-sequentially complete anti-dual pair $\dual{F}{E}$}.
Observe that  \eqref{E:ANyy-2}  provides a useful  explicit formula for the ``quadratic form'' of $A_N$. From the density of $\ran A$ in $\hila$ we can establish another one, namely, 
\begin{align*}
0&=\inf\set{\sipa{J^*y-Ax}{J^*y-Ax}}{x\in\dom A}\\
 &=\dual{JJ^*y}{y}+\inf \set[\big]{-\dual{Ax}{y}-\overline{\dual{Ax}{y}}+\dual{Ax}{x}}{x\in\dom A}
\end{align*}
for every $y\in E$,
which gives us 
\begin{align}\label{E:ANyy-1}
\dual{A_Ny}{y}&=\sup\set[\big]{\dual{Ax}{y}+\overline{\dual{Ax}{y}}-\dual{Ax}{x}}{x\in\dom A},\qquad y\in E,
\end{align}
after rearrangement.

The simple example below demonstrates that $w^*$-sequentially completeness of $F$ was really essential in the main theorem.
\begin{example}
Let $\hil$ be a Hilbert space and let $A$ be an unbounded positive self-adjoint operator in $\hil$, with (dense) domain $\dom A$. Let $E$ denote the domain of $A^{1/2}$, i.e., $E:=\dom A^{1/2}$. Then clearly, $\dual{H}{E}$ forms and anti-dual pair with respect to the duality induced by the inner product, but $H$ is not sequentially complete with respect to the corresponding weak topology. It is readily seen that $A:\dom A\to \hil$ fulfills condition (iv) of Theorem \ref{T:Krein-von Neumann}: indeed, for $y\in E$, 
\begin{align*}
\abs{\sip{Ax}{y}}^2=\abs{\sip{A^{1/2}x}{A^{1/2}y}}^2\leq \|A^{1/2}y\|^2\sip{Ax}{x}, \qquad x\in \dom A.
\end{align*}
Although condition (iv) is satisfied, the implication of Theorem \ref{T:Krein-von Neumann} does not remain true because of the lack of $w^*$-sequentially completeness. Indeed, assume that $A$ extends to a positive operator $\widetilde{ A}:E\to\hil$. Recall that a self-adjoint operator may not have any proper symmetric extension, hence $A=\widetilde{A}$ and, in particular, $\dom A=\dom A^{1/2}$. But thus is impossible because $\dom A\subsetneqq \dom A^{1/2}$ whenever $A$ is unbounded (see \cite{TZS:closedrange}*{Corollary 2.4}).
\end{example}
 
Notice that the set of positive extensions of a given positive operator $A$ has no maximal element (unless $\dom A$ is dense): for example, in the trivial case when $\dom A=\{0\}$, every positive operator is an extension of $A$. However, the next theorem says that we will get a maximum among positive extensions, bounded by a positive operator $B$. 
\begin{theorem}\label{T:max-extension}
Let $A$ be a subpositive operator on the $w^*$-sequentially complete anti-dual pair $\dual{F}{E}$. Let $B\in \lef$ be a positive operator such that $A_N\leq B$, then there exists a positive operator $\AmaxB\in\lef$, $\AmaxB\leq B$ such that for every positive extension $\widetilde A\in\lef$ of $A$, $0\leq \widetilde A\leq B$, one has $\widetilde A\leq  \AmaxB$. In other words, 
\begin{equation*}
\AmaxB=\max\set{\widetilde A\in \lef }{0\leq \widetilde A\leq B, A\subset \widetilde A}.
\end{equation*}
Furthermore, a positive operator $0\leq \widetilde A\leq B$ is an extension of $A$ if and only if $A_N\leq \widetilde A \leq \AmaxB$:
\begin{equation}\label{E:[AN,AB]}
[A_N,\AmaxB]=\set{\widetilde A\in \lef }{0\leq \widetilde A\leq B, A\subset \widetilde A}
\end{equation}
\end{theorem}
\begin{proof}
By assumption, $B-A_N\geq0$ is a positive extension of $B-A$, hence $(B-A)_N$ exists. Set 
\begin{equation*}
\AmaxB:=B-(B-A)_N.
\end{equation*}
We have apparently $0\leq \AmaxB\leq B$, and $A\subset \AmaxB$. If $0\leq \widetilde A \leq B$ extends $A$ then $B-A\subset B-\widetilde A$ implies $(B-A)_N\leq B-\widetilde A$ and hence  $ \widetilde A=B-(B-\widetilde A)\leq \AmaxB$. In order to prove \eqref{E:[AN,AB]} let $A_N\leq \widetilde A\leq \AmaxB$. Take $x\in\dom A$, then for every $y\in E$ we  have 
\begin{align*}
\abs{\dual{\widetilde Ax-Ax}{y}}^2&=\abs{\dual{(\widetilde A-A_N)x}{y}}^2\\
&\leq \dual{(\widetilde A-A_N)x}{x} \dual{(\widetilde A-A_N)y}{y} \\
&\leq \dual{(\AmaxB-A_N)x}{x} \dual{(\widetilde A-A_N)y}{y}=0,
\end{align*}
because $A_Nx=\AmaxB x=Ax$. Hence $\widetilde Ax=Ax$, that is, $A\subset \widetilde A$. 
\end{proof}

The next theorem  tells us that the Krein--von Neumann extension preserves certain commutation properties as well. We remark in advance that the proof is a bit unusual: although the claim concerns continuous operators, the proof itself  applies unbounded operator techniques.

\begin{theorem}\label{T:AN-commute}
Let $A:\dom A\to F$ be a subpositive operator on the $w^*$-sequentially complete anti-dual pair $\dual{F}{E}$.  Suppose that there are two weakly continuous operators $B,C\in \le$ leaving $\dom A$ invariant, and  that the spectrum of $BC$ restricted to $\dom A$  is bounded. Assume in addition that $B$ and $C$ satisfy
\begin{equation}\label{E:AB=CA}
C^*\!A\subset AB,\qquad\mbox{and} \qquad B^*\!A\subset AC, 
\end{equation}
then the Krein-von Neumann extension of $A$ satisfies
\begin{equation*}%\label{E:AB=CA}
C^*\!A_N= A_NB,\qquad\mbox{and} \qquad B^*\!A_N=A_NC.
\end{equation*}
\end{theorem}
\begin{proof}
Let us introduce two linear operators $\widehat{B}$ and $\widehat{C}$  on the dense linear subspace $\ran A\subseteq \hila$ by letting 
\begin{align*}
\widehat{B}(Ax):=ABx,\qquad \widehat{C}(Ax):=ACx,\qquad x\in \dom A.
\end{align*}
Observe that $\widehat{B}$ and $\widehat{C}$ satisfy
\begin{align*}
\sipa{\widehat{B}(Ax)}{Ay}=\sipa{Ax}{\widehat{C}(Ay)}, \qquad x,y\in \dom A
\end{align*} 
because from \eqref{E:AB=CA} it follows that
\begin{align*}
\sipa{ABx}{Ay}=\dual{ABx}{y}=\dual{Ax}{Cy}=\sipa{Ax}{ACy}.
\end{align*}
 We see therefore that $\widehat{B}$ and $\widehat{C}$ are closable operators and their closures satisfy 
\begin{align}\label{E:BwedgeC}
\sipa{\widehat{B}^{**}h}{k}=\sipa{h}{\widehat{C}^{**}k}, \qquad h\in\dom(\widehat B^{**}),k\in\dom (\widehat{C}^{**}).
\end{align}
We are going to prove first that $\widehat B^{**}$ and $\widehat C^{**}$ are adjoints of each other, i.e.,
\begin{align*}
\widehat C^{*}=\widehat B^{**},\qquad \widehat B^{*}=\widehat C^{**}.
\end{align*}
By  \cite{SebTarMultilin}*{Theorem 3.1} that   will be shown once we prove that  the operator matrix 
\begin{align*}
\Pi:=\begin{pmatrix}
t I_A&   -\widehat B^{**}\\
\widehat C^{**}& t I_A
\end{pmatrix}
\end{align*}
has full range, i.e., $\ran \Pi=\hila\times\hila$, for some $t\in\dupR$, $t\neq0$ (here $I_A$ stands for the identity operator of $\hila$). From \eqref{E:BwedgeC} it is easy to deduce that 
\begin{align*}
\|\Pi(h,k)\|^2\geq t^2\|(h,k)\|^2,\qquad (h,k)\in\dom \Pi,
\end{align*}
and since $\Pi$ is a closed operator, its  range is closed. On the other hand, with 
\begin{align*}
\Pi_0:=\begin{pmatrix}
t I_A&   -\widehat B\\
\widehat C& t I_A
\end{pmatrix},
\end{align*}
we have  $\Pi=\Pi_0^{**}$, hence it suffices to show that $\Pi_0$ has dense range. To do so observe that 
\begin{align*}
\begin{pmatrix}
t I_A&   -\widehat B\\
\widehat C& t I_A
\end{pmatrix}
\begin{pmatrix}
t I_A&   \widehat B\\
-\widehat C& t I_A
\end{pmatrix}=\begin{pmatrix}
t^2 I_A+\widehat{B}\widehat C&   0\\
0& t^2 I_A+\widehat{C}\widehat{B}
\end{pmatrix},
\end{align*}
and here the latter one has range 
\begin{align}\label{E:ranXXX}
\set{(A(t^2I+BC)x,A(t^2I+CB)y)}{x,y\in \dom A}\subset \hila\times\hila.
\end{align}
By assumption, we have that the restriction $t^2I+BC|_{\dom A}:\dom A\to\dom A$ is weakly continuously invertible whenever $t$ is large enough. The same holds true for $t^2I+CB$ because of the Jacobson lemma. Hence we conclude that the set in \eqref{E:ranXXX} is identical with $\ran A\times \ran A$, and therefore  $\ran A\times\ran A\subseteq \ran\Pi_0$, which proves our claim.

Next we show that $\ran J^*\subseteq \dom(\widehat B^{**})$ and $J^*B=\widehat{B}^{**}J^*$. Let $y\in E$ and $x\in \dom A$, then $\sipa{\widehat{C}(Ax)}{J^*y}=\dual{ACx}{y}=\dual{Ax}{By}=\sipa{Ax}{J^*By}$, whence we have  $J^*y\in \dom (\widehat{C }^{*})$ and $\widehat{C}^*J^*y=J^*By$. Since we have $\widehat{C}^*=\widehat{B}^{**}$, this proves $J^*B=\widehat{B}^{**}J^*$. Similarly, $J^*C=\widehat{C}J^*$. And now we have arrived: 
\begin{align*}
A_NB=JJ^*B=J\widehat{B}^{**}J^*=(J\widehat{B}^*J^*)^*=(J\widehat{C}^{**}J^*)^*=(A_NC)^*=C^*A_N,
\end{align*}
and similarly, $A_NC=B^*A_N$. The proof is complete. 
\end{proof}
We have proved above that the auxiliary operators $\widehat{B}$ and $\widehat{C}$ (more precisely, their closures) are adjoint of each other, but we do not know whether they are bounded or not. If the underlying anti-dual pair comes from a Banach space $E$ and its conjugate dual space $F=\anti{E}$ then one can imitate the proof of \cite{HSdeS}*{Theorem 2.2} to show that $\widehat{B}$, $\widehat{C}$ are bounded operators with common norm bound $\sqrt{r(BC)}$, the spectral radius of $BC$. \\

We close this section by a quick application of our main theorem. Namely, we are going to deal with so called positive definite operator functions see e.g. \cites{as,s1} for the Hilbert space case, and \cites{M,HSdS} for the Banach space setting. As these frameworks were too restrictive for various applications, many authors dealt with the case of functions taking
values in a locally convex topological vector spaces in duality (see for example \cites{AG} and the references therein). Our setting is as follows: let $\dual FE$ be a $w^*$-sequentially complete anti-dual pair and $Z$ be a non-void set. Let us denote by $\mathscr{F}$  the vector space of functions $u:Z\to E$ which have finite support, and denote by $\bar{\mathscr{F}}^*$ the algebraic conjugate dual of $\mathscr{F}$. A function $K:Z\times Z\to \lef$ is called a positive definite function (or shortly, a \emph{kernel}) if it satisfies 
\begin{equation}\label{E:kernel}
\sum_{s,t\in Z} \dual{K(s,t)u(s)}{u(t)}\geq0,\qquad \mbox{for all $u\in\mathscr{F}$}.
\end{equation} 
We can naturally associate a positive operator $A_K:\mathscr{F}\to\bar{\mathscr{F}}^*$ with $K$ by setting 
\begin{align*}
\dual{A_Ku}{v}:=\sum_{s,t\in Z} \dual{K(s,t)u(s)}{u(t)}, \qquad u,v\in \mathscr{F}.
\end{align*}
For given two kernels $K,L$ we write $K\preceq L$ if their associated positive operators satisfy $A_K\leq A_L$. Using this correspondence one can consider a kernel $K$ as subspace of $\mathscr{F}\times\bar{\mathscr{F}}^*$ of the form $\{(u,A_Ku)\,|\,u\in \mathscr{F}\}$, say the graph of $K$. In the following theorem we will characterize whether a subspace implemented by an operator $A:\mathscr{F}_0\to\bar{\mathscr{F}}^*$ can be extended to be a graph of a kernel.

\begin{theorem}
Let $\mathscr{F}_0$ be a linear subspace of $\mathscr{F}$, and assume that a positive operator $A:\mathscr{F}_0\to\bar{\mathscr{F}}^*$ is given. Then the following assertions are equivalent:
\begin{enumerate}[\upshape (i)]
 \item There exists a kernel $K:Z\times Z\to \lef$ so that $A\subseteq A_K$,
 \item For every $v\in\mathscr{F}$ there is $M_v\geq0$ such that 
 \begin{align*}
 \abs{\dual{Au}{v}}^2\leq M_v\cdot\dual{Au}{u}\qquad\mbox{for all $u\in\mathscr{F}_0$}.
 \end{align*}
 Moreover, there is a $K$ that is smallest in the sense that $K\preceq L$ for any kernel $L$ with $A\subseteq A_L$.
\end{enumerate}
\end{theorem}
\begin{proof}
It is obvious that (i) implies (ii). For the converse observe that $\mathscr{F}$ and $\bar{\mathscr{F}}^*$ intertwined by the evaluation form a $w^*$-sequentially complete anti-dual pair, thus we can apply Theorem \ref{T:Krein-von Neumann}. Therefore the smallest (Krein-von Neumann) extension $A_N$  of $A$ exists. Our only duty is  to show that $A_N$ is induced by a suitable kernel $K$. To this aim, for any fixed $s\in Z$ and $x\in E$ introduce the function $u_{s,x}\in\mathscr F$ by setting 
\begin{align*}
u_{s,x}(s'):=\delta_s(s')\cdot x,\qquad s'\in Z,
\end{align*}
where $\delta_s$ is the usual Dirac-function concentrated at $s$. For $s,t\in Z$ define a linear operator $K(s,t):E\to F$ by 
\begin{align*}
\dual{K(s,t)x}{y}:=\dual{A_Nu_{s,x}}{u_{t,y}},\qquad x,y\in E.
\end{align*}
By the symmetry of $A_N$ we conclude that 
\begin{align*}
\dual{K(s,t)x}{y}=\dual{A_Nu_{s,x}}{u_{t,y}}=\overline{\dual{A_Nu_{t,y}}{u_{s,x}}}=\overline{\dual{K(t,s)y}{x}},
\end{align*}
whence it is clear that $K(s,t)\in\lef$ (and also that $K(s,t)^*=K(t,s)$). Finally, since every $u\in\mathscr F$ can be expressed as $u=\sum\limits_{s\in Z}u_{s,u(s)}$ we have 
\begin{align*}
\dual{A_Nu}{v}=\sum_{s,t\in Z}\dual{Au_{s,u(s)}}{u_{t,v(t)}}=\sum_{s,t\in Z}\dual{K(s,t)u(s)}{v(t)},
\end{align*}
which means that $A_N$ is induced by the kernel $K$.
\end{proof}

\section{The Banach space setting}\label{S: Banach}

In this section we are going to investigate the special case when the anti-duality is the evaluation on the pair of a fixed Banach space $E$ and its conjugate topological dual $\anti{E}$.  We will obtain a strengthening of the main result of \cite{SSZT}, and we will show that Halmos' original result is indeed a corollary of our main theorem.

We remark that the Banach-Steinhaus theorem forces $\anti{E}$ to be weakly sequentially complete, and hence everything has been proved in the preceding sections remains valid also for the anti-dual pair $\dual{\anti{E}}{E}$.

We start our investigations with a suitable generalization of the Schwarz inequality (cf. Kadison \cite{Kadison}*{Theorem 1}, and also  Lemma 2.1 and Corollary 2.2 in \cite{TarcsayBanach}).
\begin{proposition} Let $E$ be a Banach space and let $A_i:E\to\anti E$ $(1\leq i \leq n)$ be positive operators. Then for every set of vectors $x_i\in E$ ${1\leq i\leq n}$ the following inequality holds

\begin{align}\label{E:genSchwarz}
\Big\|\sum_{j=1}^n A_j x_j\Big\|^2\leq\Big\|\sum_{j=1}^n A_j \Big\|\cdot \sum_{j=1}^n \dual{A_j x_j}{x_j}.
\end{align}
Moreover, $C:= \|\sum_{j=1}^n A_j\|$ is the smallest constant satisfying \eqref{E:genSchwarz}.
\end{proposition}
\begin{proof}
Consider the direct sum Hilbert-space  $\kil:=\bigoplus_{j=1}^n \hil_{A_j}$, and define a linear operator $V:E\to\kil$ by $Vx:=(A_1x,A_2x,\ldots,A_nx)$ for $x\in E$. It is easy to check that $V$ is continuous by norm bound $\|V\|^2\leq \|\sum\limits_{j=1}^n A_j \|$. A straightforward calculation shows that its adjoint acts as 
$V^*(A_1x_1,\ldots,A_nx_n):=\sum_{j=1}^n A_jx_j$, for all $x_1,x_2,\ldots,x_n\in E$. Since $\|V^*\|=\|V\|$, we infer that 
\begin{align*}
\Big\|\sum_{j=1}^n A_jx_j\Big\|^2&=\|V^*(A_1x_1,A_2x_2,\ldots,A_nx_n)\|^2\\
&\leq \|V^*\|^2\cdot \|(A_1x_1,A_2x_2,\ldots,A_nx_n)\|^2\\
&\leq \Big\|\sum_{j=1}^n A_j\Big\|\cdot\sum_{j=1}^n\dual{A_jx_j}{x_j}.
\end{align*} In order to prove minimality, let us fix a constant $C\geq0$ that satisfies \eqref{E:genSchwarz} with $\|\sum_{j=1}^n A_j \|$ replaced by $C$. Then for $x\in E$ the following inequalities hold
\begin{align*}
\Big\|\sum_{j=1}^n A_jx\Big\|^2\leq C\cdot\sum_{j=1}^n\dual{A_jx}{x} \leq C\cdot \Big\|\sum_{j=1}^n A_j\Big\|\cdot \|x\|^2,
\end{align*}
thus we have $\|\sum_{j=1}^n A_j \|\leq C$, indeed.
\end{proof}
As an immediate consequence of the preceding proposition ($n=1$) we have 
\begin{align}\label{C:Schwarz} 
\|Ax\|^2\leq \|A\|\dual{Ax}{x},\qquad x\in E.
\end{align}
Now we rephrase our main theorem in the Banach space setting, extending the results obtained in \cite{SSZT}*{Theorem 3.1} and also \cite{ATV}*{Theorem 2.2}.
\begin{theorem}\label{T:Banach-extension}
Let $E$ be a Banach space, and let $A:E\supseteq \dom A\to \anti{E}$ be a linear operator. Then the following statements are equivalent.

\begin{enumerate}[\upshape (i)]
 \item $A$ has a (continuous) positive extension $\widetilde{A}\in\lee$,
 \item There is a constant $M\geq0$ such that 
 \begin{align}\label{E:AxleqAxx}
  \|Ax\|^2\leq M\cdot\dual{Ax}{x},\qquad x\in\dom A, 
 \end{align}
 \item For any $y\in \hil$ there exists $M_y\geq0$ such that 
 \begin{align*}
    \abs{\dual{Ax}{y}}^{2}\leq M_y\cdot\dual{Ax}{x},\qquad x\in\dom A.
 \end{align*}
\end{enumerate}
In any case, there exists the Krein-von Neumann extension $A_N$ of $A$ that is the smallest among the set of positive extensions of $A$. The norm of $A_N$ satisfies
\begin{equation}\label{E:AN-norm}
\|A_N\|=\inf\set{M\geq0}{\|Ax\|^2\leq M\cdot \dual{Ax}{x},\quad x\in\dom A}.
\end{equation}
 If $B,C\in \le$ are continuous operators leaving $\dom A$ invariant such that $C^*\!A\subset AB$ and $B^*\!A\subset AC$ then 
the Krein-von Neumann extension of $A$ satisfies
\begin{equation}\label{E:AB=CA-Banach}
C^*\!A_N= A_NB,\qquad\mbox{and} \qquad B^*\!A_N=A_NC.
\end{equation}
\end{theorem}
\begin{proof}
The equivalence between (i) and (iii) is clear from Theorem \ref{T:Krein-von Neumann}. That (i) implies (ii) follows from  \eqref{C:Schwarz} with $M=\|\widetilde A\|$. To see that (ii) implies (iii), fix $y\in E$ and observe that 
\begin{align*}
 \abs{\dual{Ax}{y}}^{2}\leq \|Ax\|^2 \|y\|^2\leq M\cdot\|y\|^2\cdot\dual{Ax}{x},\qquad x\in\dom A.
\end{align*}
Hence (iii) holds with $M_y=M\cdot \|y\|^2$.  

To prove \eqref{E:AN-norm}, let $M_A$ denote the corresponding infimum. It follows from \eqref{C:Schwarz} that $M_A\leq \|A_N\|$. To see the converse, take $y\in E$, $\|y\|\leq 1$, then for every $x\in\dom A$, 
\begin{align*}
\abs{\dual{Ax}{y}}^2\leq \|Ax\|^2\leq M\cdot\dual{Ax}{x},
\end{align*}
hence $\dual{Ay}{y}\leq M$, according to formula \eqref{E:ANyy-2}. Consequently, $\|A_N\|^2\leq M^2$ follows from
\begin{align*}
\sup_{\substack{y,z\in E,\\ \|y\|=\|z\|=1}} \abs{\dual{A_Ny}{z}}^2\leq \sup_{\substack{y,z\in E,\\ \|y\|=\|z\|=1}} \dual{A_Ny}{y}\dual{A_Nz}{z}\leq M^2.
\end{align*}
The rest of the statement follows immediately from Theorem \ref{T:AN-commute}.
\end{proof}
We remark here that implication (ii)$\Rightarrow$(iii) could be deduced also from (iii)$\Rightarrow$(iv) of Theorem \ref{T:Krein-von Neumann} because inequality \eqref{E:AxleqAxx} expresses just that $W(A)$ is bounded in the functional norm, i.e., $W(A)$ is $\beta(\anti{E},E)$-bounded. 

As an immediate consequence of Theorem \ref{T:Banach-extension} we recover \cite{Sebestyen93}*{Theorem 1} concerning positive extendibility of suboperators on a Hilbert space. 
\begin{corollary}\label{T:Hilbert-extension}
Let $\hil$ be a Hilbert space and let $A:\hil\supseteq \dom A\to \hil$ be a linear operator. Then the following assertions are equivalent:
\begin{enumerate}[\upshape (i)]
 \item $A$ has a (continuous) positive extension $\widetilde{A}\in\bh$,
 \item There is a constant $M\geq0$ such that 
 \begin{align*}%\label{E:AxleqAxx}
  \|Ax\|^2\leq M\cdot\sip{Ax}{x},\qquad x\in\dom A, 
 \end{align*}
 \item For any $y\in \hil$ there exists $M_y\geq0$ such that 
 \begin{align*}
    \abs{\sip{Ax}{y}}^{2}\leq M_y\cdot\sip{Ax}{x},\qquad x\in\dom A.
 \end{align*}
\end{enumerate}
 In any case, there exists is the Krein-von Neumann extension $A_N$ of $A$  whose norm satisfies
 \begin{equation*}%\label{E:AN-norm}
\|A_N\|=\inf\set{M\geq0}{\|Ax\|^2\leq M\cdot \sip{Ax}{x},\quad x\in\dom A}
\end{equation*}
 If $B,C\in \bh$ are bounded operators leaving $\dom A$ invariant such that $C^*\!A\subset AB$ and $B^*\!A\subset AC$ then 
the Krein-von Neumann extension of $A$ satisfies
\begin{equation*}%\label{E:AB=CA-Banach}
C^*\!A_N= A_NB,\qquad\mbox{and} \qquad B^*\!A_N=A_NC.
\end{equation*}
\end{corollary}
Note that for a positive operator $A\in\bh$ one has $\|A\|\leq M$ if and only if $A\leq M \cdot I$, $I$ being the identity operator of $\hil$. Combining this observation with Theorem \ref{T:max-extension} one immediately gets the following statement.
\begin{corollary}
Let $A:\dom A\to\hil$ be a positive operator  satisfying  the conditions of Corollary \ref{T:Hilbert-extension}. Then, for every constant $M\geq\|A_N\|$ there is a positive extension $\AmaxM$ of $A$ with $\|\AmaxM\|\leq M$ such that  for any positive extension $\widetilde A$ of $A$, $\|\widetilde A\|\leq M$ one has $\widetilde A\leq \AmaxM$. In  other words,
\begin{equation*}
\AmaxM=\max\set{\widetilde A\in\bh}{\widetilde A\geq0, A \subset \widetilde A, \|\widetilde A\|\leq M}.
\end{equation*}
Furthermore, one has equality 
\begin{equation*}
[A_N,\AmaxM]=\set{\widetilde A\in\bh}{\widetilde A\geq0, A \subset \widetilde A, \|\widetilde A\|\leq M}.
\end{equation*}
\end{corollary}

As a concluding result of this section we prove Halmos'  result on positive completions of incomplete matrices (see \cite{Halmos}*{\S 5, Corollary 2}). 
 \begin{corollary}\label{C: Halmos-recover}
 Let $A:=\opmatrix{A_{11}}{A_{12}}{A_{21}}{*}$ be an ``incomplete'' operator matrix in a Hilbert space $H=H_0\oplus H_0^{\perp}$ with $\dom A=H_0$. Assume that $A_11\geq0$ and that $A_{21}=A_{21}^*$, then the following assertions are equivalent:
 \begin{enumerate}[\upshape (i)]
 \item there exists $A_{22}\geq0$ that makes $A$ positive, i.e., $\opmatrix{A_{11}}{A_{12}}{A_{21}}{A_{22}}\geq0$,
 \item $A_{21}^*A_{21}\leq M\cdot A_{11}$, for some constant $M\geq0$,
 \item $\ran A_{21}^*\subseteq \ran A_{11}^{1/2}$.
 \end{enumerate}
 \end{corollary}
\begin{proof}
Assume first (i) and prove (ii): Suppose that $B:=\opmatrix{A_{11}}{A_{12}}{A_{21}}{A_{22}}\geq0$. Since for every positive operator $B$ one has $B^2\leq \|B\|\cdot B$, we have for $x\in H_0$ that 
\begin{align*}
\|A_{11}x\|^2+\|A_{21}x\|^2&=\left\|B\pair x0\right\|^2\leq \|B\| \sip[\bigg]{B\pair x0}{\pair x0}=\|B\|\sip{A_{11}x}{x}.
\end{align*}
Since $A_{11}^2\leq \|A_{11}\|\cdot A_{11}$, this obviously gives (ii).

The equivalence of (ii) and (iii) follows immediately from the Douglas factorization theorem \cite{Douglas}. Finally, assuming (ii) we clearly have 
\begin{align*}
\left\|A\pair x0\right\|^2\leq \|A_{11}\|\sip{A_{11}x}{x}+M\sip{A_{11}x}{x}, \qquad x\in H_0,
\end{align*}
hence $A$ fulfills (ii) of Corollary \ref{T:Hilbert-extension}. Consequently, $A$ has a bounded positive  extension, or equivalently, $A$ can be filled in to get a positive operator. 
\end{proof}
Finally we mention that Friedrichs extension of unbounded operators acting between reflexive Banach spaces and their duals has been investigated by Farkas and Matolcsi in \cite{Farkas-Matolcsi}.

\section{Functional extensions on *-algebras}\label{S: funct}
The goal of this section is to show that how abstract operator extensions can be applied for positive functionals on a *-algebra.

Let $\alg$ be a not necessarily unital complex $^*$-algebra. Recall that a linear functional $f$ on $\alg$ is called \emph{representable} if there exist a Hilbert space $\hil$, a *-representation  (that is, a *-homomorphism) $\pi:\mathscr{A}\to\bh$ and  a vector $\xi\in\hil$ such that
$$f(a)=\sip{\pi(a)\xi}{\xi},\qquad a\in\mathscr{A}.$$
 It is clear that every representable functional is positive, nevertheless the converse is not true in such generality.
 
Consider now a left ideal $\M$ of $\alg$ and  a linear functional $f:\M\to\dupC$. In this section we provide  necessary and sufficient conditions under which $f$ admits a representable extension to $\alg$ (cf. also \cite{SSZT} for the Banach-* algebra setting). Recall that if $f:\M\to\dupC$ is a linear functional, then we can associate an operator $A:\M\to \bar\alg^*$ to $f$ by setting 
\begin{equation}\label{E:A_from_f}
 \dual{Aa}{x}:=f(x^*a),\qquad x\in \alg, a\in \M.
\end{equation}
Clearly, $A$ is positive if and only if $f$ is positive, i.e., $f(a^*a)\geq0$ holds for all $a\in\M$. 

In what follows we are exclusively interested in representable extensions of functionals. 
Recall that $f$ is said to be \emph{Hilbert bounded} if there is constant $M\geq 0$ such that 
\begin{equation}\label{E:Hilbertbounded}
 \abs{f(a)}^2\leq Mf(a^*a)\qquad \textrm{for all $a\in \M$},
\end{equation}
and \emph{admissible} if for any $x$ in $\alg$ there exists $\lambda_x\geq 0$ such that 
 \begin{equation}\label{E:admissible}
    f(a^*x^*xa)\leq \lambda_xf(a^*a)\qquad \textrm{for all $a\in\M$.}
 \end{equation}

\begin{theorem}\label{T:functext}
Let $\alg$ be a ${}^*$-algebra, $\M\subseteq\alg$ be a left ideal, and let $f:\M\to\dupC$ be a linear functional. The following assertions are equivalent:
\begin{enumerate}[\upshape (i)]
 \item there is a representable functional $\widetilde{f}\in\alg^*$ extending $f$,
 \item $f$ is  admissible and Hilbert bounded.
\end{enumerate} 
If there is any, then there is a minimal one (denoted by $f_N$) among the set of representable extensions.   
\end{theorem}
\begin{proof}
Implication  (i)$\Rightarrow$(ii) is trivial so we only have to prove   that if a functional $f$ on a left ideal $\ideal$ of $\alg$ satisfies \eqref{E:Hilbertbounded} and \eqref{E:admissible} then  $f$ admits a minimal representable extension. We are going to imitate the proof of \cite{Sebestyen84}*{Theorem 1}, combined with the construction of the proof of Theorem \ref{T:Krein-von Neumann}: suppose that  $f$ is admissible and Hilbert bounded and consider the  positive operator $A:\ideal\to\anti\alg$, defined by \eqref{E:A_from_f}.
Since the following chain of inequalities holds true for all $a\in\M$ and $x\in\alg$
\begin{align*}
\abs{\dual{Aa}{x}}^2=\abs{f(x^*a)}^2\leq Mf(a^*xx^*a)\leq \lambda_{x^*}M f(a^*a)=\lambda_{x^*} M\dual{Aa}{a},
\end{align*}
we conclude that $A$ fulfills condition (iv) of Theorem \ref{T:Krein-von Neumann}, and thus the Krein-von Neumann extension $A_N=JJ^*:\alg\to\bar\alg^*$ exists. Consider the auxiliary Hilbert space $\hila$, and for every $x\in \alg$ define the operator $\pi_f(x)\in\bha$ as the continuous extension of the densely defined bounded operator 
\begin{align*}
\pi_f(x)(Aa):=A(xa),\qquad a\in\ideal.
\end{align*}
Note that boundedness of $\pi_f(x)$ is guaranteed by admissibility \eqref{E:admissible}: 
\begin{equation*}
\|\pi_f(x)(Aa)\|_A^2=\dual{A(xa)}{xa}=f(a^*x^*xa)\leq \lambda_xf(a^*a)=\lambda_x\|Aa\|_A^2.
\end{equation*}
We claim that $\pi_f:\alg\to\bha$ is a *-homomorphism. It is straightforward that $\pi_f$ is linear and multiplicative. To see that $\pi_f$ preserves involution, fix $x\in\alg$ and $a,b\in\M$, then
\begin{align*}
\sipa{\pi_f(x)(Aa)}{Ab}&=\dual{A(xa)}{b}=f(b^*xa)\\ &=f((x^*b)^*a)=\dual{Aa}{x^*b}=\sipa{Aa}{\pi_f(x^*)(Ab)},
\end{align*}
hence $\pi_f(x)^*=\pi_f(x^*)$, indeed. Our next goal is to show the existence of a vector $\zeta_f\in\hila$ such that  $f(a)=\sipa{\pi_f(a)\zeta_f}{\zeta_f}$, $a\in\M$. To this aim consider  the densely defined linear functional
\begin{equation}\label{E:phi(Ax)=f(a)}
\phi:H_A\supseteq\ran A\to\dupC;\quad \phi(Aa)=f(a),\qquad a\in\M,
\end{equation}
which  is well defined and continuous because of Hilbert boundedness \eqref{E:Hilbertbounded}:
\begin{equation*}
\abs{\phi(Aa)}^2\leq M f(a^*a)=M\|Aa\|^2_A,\qquad a\in\M.
\end{equation*}
By the Riesz representation theorem, $\phi$ is represented by a unique vector  $\zeta_f\in\hila$:
\begin{equation}\label{E:zetaf-def}
f(a)=\sipa{Aa}{\zeta_f},\qquad a\in\M.
\end{equation}
We claim that 
\begin{equation}\label{E:fN-def}
f_N(x):=\overline{\dual{J\zeta_f}{x}},\qquad x\in\alg
\end{equation}
defines a representable functional  such that $f\subset f_N$. To see this, observe first that the following useful identity holds true
\begin{equation}\label{E:pixzeta}
\pi_f(x)\zeta_f=J^*x,\qquad x\in\alg.
\end{equation}
Indeed, let $x\in \alg$ and $a\in\M$ be fixed. Then by \eqref{E:zetaf-def} we have
\begin{align*}
\sipa{Aa}{\pi_f(x)\zeta_f}=\sipa{A(x^*a)}{\zeta_f}&=f(x^*a)=\dual{Aa}{x}=\dual{J(Aa)}{x}=\sipa{Aa}{J^*x}.
\end{align*}
Consequently, $f(a)=\sipa{Aa}{\zeta_f}=\sipa{J^*a}{\zeta_f}=\overline{\dual{J\zeta_f}{a}}=f_N(a)$
holds for $a\in\M$, thus $f\subset f_N$. On the other hand we have  
\begin{equation*}
\sipa{\pi_f(x)\zeta_f}{\zeta_f}=\sipa{J^*x}{\zeta_f}=\overline{\dual{J\zeta_f}{x}}=f_N(x),\qquad x\in\alg,
\end{equation*}
hence $f_N$ is representable. All that remains is to prove the minimality of  $f_N$. Let us calculate first the value of $f_N$ on positive elements:
\begin{align*}
f_N(x^*x)&=\overline{\dual{J\zeta_f}{x^*x}}=\sipa{J^*(x^*x)}{\zeta_f}=\sipa{J^*x}{J^*x}\\
&=\sup\set{\abs{\sipa{Aa}{J^*x}}^2}{a\in\M, \|Aa\|^2_A\leq 1}\\
&=\sup\set{\abs{f(x^*a)}^2}{a\in\M, f(a^*a)\leq 1},
\end{align*}
for every $x\in\alg$, where we used the density of $\ran A$ in $\hila.$
Let now $\widetilde f$ be any representable extension of $f$. Using identity $\widetilde f=\widetilde f_N$, we conclude that 
\begin{align*}
\widetilde f(x^*x)&=\sup\set{\abs{\widetilde f(x^*a)}^2}{a\in\alg, \widetilde f(a^*a)\leq 1}\\
&\geq \sup\set{\abs{f(x^*a)}^2}{a\in\M, f(a^*a)\leq 1}=f_N(x^*x),
\end{align*}
hence $f_N\leq \widetilde f$. The proof is complete.
\end{proof}
Following the terminology of \cite{SSZT}, we will call $f_N$ the Krein--von Neumann extension of $f$. Note that in the above construction we gained an explicit formula for  $f_N$ and also for its values on positive elements.
\begin{corollary}
Suppose that $f:\M\to\dupC$ is admissible and Hilbert-bounded, then the minimal extension $f_N$ satisfies 
\begin{equation}\label{E:Jzetaf}
f_N(x)=\overline{\dual{J\zeta_f}{x}},\qquad x\in\alg,
\end{equation}
and 
\begin{equation}\label{E:f(x*x)}
f_N(x^*x)= \sup\set{\abs{f(x^*a)}^2}{a\in\M, f(a^*a)\leq 1},
\end{equation}
\end{corollary}
In case when the algebra has a unit element the theorem can be more easily  formulated. In fact, the following simple formula may suggest that extension theory of functionals and extension theory of operators fit nicely together, indeed.
\begin{corollary}
Assume that $\alg$ is a unital ${}^*$-algebra with unit $1\in\alg$. Assume further that $f:\M\to\dupC$ is admissible. Then $f$ is Hilbert bounded, and its Krein-von Neumann extension satisfies
\begin{equation}\label{E:AN1}
f_N(x)= \overline{\dual{A_N1}{x}},\qquad x\in \alg.
\end{equation}
\end{corollary}
\begin{proof}
Take $a\in\M$, then 
\begin{equation*}
\abs{f(a)}^2=\abs{f(1^*a)}^2\leq f(1)f(a^*a),
\end{equation*}
hence $f$ is Hilbert  bounded. On the other hand, for $a\in\M$, 
\begin{equation*}
\sipa{Aa}{J^*1}=\dual{Aa}{1}=f(a)=\sipa{Aa}{\zeta_f},
\end{equation*}
hence $J^*1=\zeta_f$. This yields
\begin{align*}
\dual{A_N1}{x}=\sipa{J^*1}{J^*x}=\sipa{\zeta_f}{J^*x}=\dual{J\zeta_f}{x},\qquad x\in\alg,
\end{align*}
and hence \eqref{E:AN1}.
\end{proof} 
In the next theorem we provide the analogue of Theorem \ref{T:max-extension}.
\begin{theorem}
Let $f:\M\to\dupC$ be an admissible and Hilbert bounded functional and fix any  representable functional $g\in\alg^*$ such that $f_N\leq g$. Then there is a representable extension $f^g_{max}\in\alg^*$ of $f$ such that $f^g_{max}\leq g$, and for every representable extension $\widetilde f$ of $f$ one has $\widetilde f\leq f^g_{max}$. In other words, 
\begin{align*}
f^g_{max}=\max\set{\widetilde f\in\alg^{\sharp}}{\widetilde f\leq g, f\subset \widetilde f}.
\end{align*}
Furthermore, a representable  functional $\widetilde f\leq g$ is an extension of $f$ if and only if $f_N\leq \widetilde f\leq f^g_{max}$:
\begin{equation*}
[f_N, f^g_{max}]=\set{\widetilde f\in\alg^{\sharp}}{\widetilde f\leq g, f\subset \widetilde f}
\end{equation*}
\end{theorem}
\begin{proof}
By assumption, $0\leq g-f_N\leq g$, hence $g-f_N$ is representable due to \cite{palmer}*{Proposition 9.4.22}. Clearly, $g-f\subset g-f_N$, hence $(g-f)_N$ exists. Set 
\begin{equation*}
f^g_{max}:=g-(g-f)_N,
\end{equation*}
which is again representable by \cite{palmer}*{Proposition 9.4.22}. One can follow the argument of the proof Theorem \ref{T:max-extension} to show that  $f^g_{max}$ possesses every claimed properties.

Finally, let $\widetilde f\in[f_N,f^g_{max}]$, then for $a\in\M$, 
\begin{align*}
\abs{(\widetilde f-f)(a)}^2&=\abs{(\widetilde f-f_N)(a)}^2\leq M' (\widetilde f-f_N)(a^* a)\leq M' ( f^g_{max}-f_N)(a^* a)=0
\end{align*}
for some $M'\geq0$, because of Hilbert boundedness. Hence $f\subset \widetilde f$, which completes the proof.
\end{proof}

\end{document}